\documentclass[12pt]{amsart}

\usepackage{amssymb}
 \usepackage{color}

\newtheorem{theorem}{Theorem}[section]

\newtheorem{corollary}[theorem]{Corollary}
\newtheorem{proposition}[theorem]{Proposition}

\theoremstyle{definition}

\newtheorem{example}[theorem]{Example}

\theoremstyle{remark}
\newtheorem{remark}[theorem]{Remark}

\numberwithin{equation}{section}

\newcommand{\tp}{\overset{\!\!\!\!\!\circ}{\sT_+}}
\newcommand{\ind}{\mathrm{ind}\,}
\newcommand{\diag}{\mathrm{diag}\,}
\newcommand{\wind}{\mathrm{wind}\,}
 
 \newcommand{\vp}{\varphi}
\newcommand{\D}{\displaystyle}

\newcommand{\smb}{\mbox{\rm smb}\,}
\newcommand{\coker}{\mbox{\rm coker}\,}

\renewcommand{\Im}{\mathrm{Im}\,}
\newcommand{\RE}{\mbox{\rm Re }}
\newcommand{\im}{\mathrm{im}\,}

\newcommand{\nn}{\nonumber}

\newcommand{\re}{\textrm{Re}\,}

 \newcommand{\esssup}{\hbox{\rm ess}\sup_{\!\!\!\!\!\!\!\!\!\!\! t\in {\mathbb{T}}}}

\newcommand{\cA}{\mathcal{A}}
\newcommand{\cB}{\mathcal{B}}
\newcommand{\cC}{\mathcal{C}}

\newcommand{\cK}{\mathcal{K}}
\newcommand{\cL}{\mathcal{L}}
\newcommand{\cM}{\mathcal{M}}

\newcommand{\cR}{\mathcal{R}}

\newcommand{\sC}{{\mathbb C}}

\newcommand{\sR}{{\mathbb R}}

\newcommand{\sT}{{\mathbb T}}

\newcommand{\sZ}{{\mathbb Z}}

\begin{document}

  \begin{center}
{\large\sc SOME RESULTS ON THE INVERTIBILITY OF\\[1ex] TOEPLITZ PLUS
HANKEL OPERATORS}

\vspace{5mm}

\textbf{Victor D. Didenko and Bernd Silbermann}\footnote{This
research was carried out when the second author visited the
Universiti Brunei
 Darussalam (UBD). The support of UBD provided via Grant UBD/GSR/S\&T/19  is highly
 appreciated.}


\vspace{2mm}


 Universiti Brunei Darussalam,
Bandar Seri Begawan, BE1410  Brunei; diviol@gmail.com

Technische Universit{\"a}t Chemnitz, Fakult{\"a}t Mathematik, 09107
Chemnitz, Germany; silbermn@mathematik.tu-chemnitz.de

 \end{center}


\date{}



\vspace{5mm}

{\footnotesize\textbf{Abstract.}
  The paper deals with the invertibility of Toeplitz plus Hankel
 operators $T(a)+H(b)$ acting on classical Hardy spaces on the unit
 circle $\sT$. It is supposed that the generating functions $a$ and $b$
 satisfy the condition $a(t)a(1/t)=b(t)b(1/t)$, $t\in\sT$. Special
 attention is paid to the case of piecewise continuous generating
 functions. In some cases the dimensions of null spaces of the operator $T(a)+H(b)$
 and its adjoint are described.}

\vspace{5mm}


 \section{Introduction}

\renewcommand{\labelenumi}{(\roman{enumi})}
Fredholm properties of Toeplitz plus Hankel  operators $T(a)+H(b)$
with piecewise continuous generating functions $a$ and $b$ have been
studied for many years. These operators are considered in various
Banach and Hilbert spaces, and the results obtained show that the
structure of the algebras generated by such operators is much more
complicated than the structure of the algebras generated by
one-dimensional singular integral operators with piecewise
continuous coefficients defined on closed smooth curves. Moreover,
calculating the index of Toeplitz plus Hankel operators, one
encounters even more difficult problems, and the difficulties grow
if one attempts to study invertibility or one-sided invertibility of
such operators. It is worth noting that, in general, Fredholm
operators $T(a)+H(b)$ are not one-sided invertible. Nevertheless, a
few works where one-sided invertibility was discussed, have appeared
in literature recently. They are mainly concerned with two special
cases of Toeplitz plus Hankel operators -- viz. with the operators
having the form $M(a):=T(a)+H(a)$ or the form
$\widetilde{M}(a):=T(a)+H(\widetilde{a})$, where
$\widetilde{a}(t):=a(1/t)$, $t\in\sT$ and   $\sT :=\{t \in \sC:
|t|=1 \}$ is the counterclockwise oriented unit circle
\cite{BE2004,CS2010a,CS2010b}. Thus one-sided invertibility of
Fredholm operators $M(a)$ was established in \cite{BE2004}. Similar
problems are studied in \cite{CS2010a, CS2010b} but a different
approach is employed. Moreover, in order to describe the kernel and
cokernel dimensions of Fredholm Toeplitz plus Hankel operators
$T(a)+H(b)$ the method of asymmetric factorization has been proposed
in \cite{BE2004, E2004}. This method is similar to the Wiener-Hopf
factorization used in the theory of Toeplitz operators. Still, for
$a\neq b$  the papers \cite{BE2004, E2004} do not offer effectively
verifiable conditions of one-sided invertibility of the operators
under consideration.

The present paper deals with Hankel plus Toeplitz operators
$T(a)+H(b)$, the generating functions of which satisfy the condition
\begin{equation*}
a \widetilde{a}=b \widetilde{b}.
\end{equation*}
The focus here is mainly on piecewise continuous generating
functions but certain results remain true for general generating
functions. Moreover, the approach used in this paper is not limited
to Toeplitz plus Hankel operators considered below. The results
remains true for Toeplitz plus Hankel operators with generating
functions $a,b\in PC_p$ and acting on the spaces $l^p(\sZ)$,
$1<p<\infty$, and also for Wiener-Hopf plus Hankel operators acting
on $L^p(\sR_+)$. The corresponding operators acting on weighted
spaces with appropriately chosen weights can be studied, as well.
Nevertheless, for the sake of simplicity, here we only consider the
case of $H^p$ spaces.

\section{Spaces and operators \label{l1}}

Let $X$ be a Banach space. By $\cL(X)$ we denote the Banach algebra
of all linear continuous operators on $X$. An operator $A\in\cL(X)$
is called Fredholm if the range $\im A:=\{ Ax:x\in X\}$ of the
operator $A$ is a closed subset of $X$ and the null spaces $\ker
A:=\{x\in X: Ax=0\}$ and $\ker A^*:=\{h\in X^*: A^*h=0\}$ of the
operator $A$ and the adjoint operator $A^*$ are finite-dimensional.
For the sake of convenience, the null space of the adjoint operator
$A^*$ is called the cokernel of $A$ and it is denoted by $\coker A$.
Further, if an operator $A\in\cL(X)$ is Fredholm, then the number
 \begin{equation*}
\kappa:=\dim\ker A-\dim\coker A
\end{equation*}
is referred to as the index of the operator $A$. Note that by $\dim
Y$ we denote the dimension of the linear space $Y$. Let $\cK(X)$
denote the set of all compact operators from $\cL(X)$. Then the
Fredholmness of an operator $A\in\cL(X)$ is equivalent to the
invertibility of the coset $A+\cK(X)$ in the Calkin algebra
$\cL(X)/\cK(X)$.

Let us now introduce some operators and spaces we need. As usual,
let $L^\infty(\sT)$ stand for the $C^*$-algebra of all essentially
bounded Lebesgue measurable functions on $\sT$, and let
$L^p=L^p(\sT)$, $1\leq p \leq\infty$ denote the Banach space of all
Lebesgue measurable functions $f$ such that

 \begin{align*}
||f||_p:= &\left ( \int_\sT |f(t)|^p \, dt \right )^{1/p}, \quad
1\leq p <\infty, \\
 ||f||_\infty:  = & \,\, \esssup |f(t)|,
\end{align*}
is finite. Further, let $H^p$ and $\overline{H^p}$ refer to the
Hardy spaces of all functions $f\in L^p$  the Fourier coefficients
of which
 \begin{equation*}
f_n=\frac{1}{2\pi}\int_0^{2\pi}f(e^{i\theta}) e^{-in\theta}\,
d\theta
 \end{equation*}
vanish for all $n<0$ and $n>0$, respectively. It is a classical
result that for $p\in (1,\infty)$ the Riesz projection $P$, defined
by
 \begin{equation*}
P: \sum_{n=-\infty}^{\infty} f_n e^{in\theta} \mapsto
\sum_{n=0}^{\infty} f_n e^{in\theta},
 \end{equation*}
is bounded on the space $L^p$ and its range is the whole space
$H^p$. The operator $Q:=I-P$,
 \begin{equation*}
Q:\sum_{n=-\infty}^{\infty} f_n e^{in\theta} \mapsto
\sum_{n=-\infty}^{-1} f_n e^{in\theta}
 \end{equation*}
is also a projection and its range is a subspace of the codimension
one in $\overline{H^p}$.

We also consider the flip operator $J:L^p\mapsto L^p$,
 \begin{equation*}
(Jf)(t):=\overline{t}f(\overline{t}), \quad t\in \sT,
 \end{equation*}
where the bar denotes the complex conjugation. Note that the
operator $J$ changes the orientation, satisfies the relations
 \begin{equation*}
J^2=I, \quad  JPJ=Q, \quad JQJ=P,
 \end{equation*}
and for any $a\in L^\infty$,
 \begin{equation*}
JaJ=\widetilde{a}I.
 \end{equation*}

Now let us introduce Toeplitz and Hankel operators $T(a)$ and
$H(a)$.

 The operator $T(a):H^p\mapsto H^p$ is defined
for all $a\in L^\infty$ and $1<p<\infty$ by
 \begin{equation*}
T(a): f\mapsto Paf.
 \end{equation*}
This operator is obviously bounded and
 \begin{equation*}
||T(a)||\leq c_p ||a||_\infty,
 \end{equation*}
where $c_p$ is the norm of the Riesz projection on $L^p$. This
operator is called Toeplitz operator generated by the function $a$.
Toeplitz operators with matrix-valued generating functions acting on
$H^p\times H^p$ are defined similarly.

For $a\in L^\infty$,  the Hankel operator $H(a):H^p\mapsto H^p$ is
defined by
 \begin{equation*}
H(a): f\mapsto PaQJf.
 \end{equation*}
It is clear that this operator is also bounded but, in contrast to
Toeplitz operators, the corresponding generating function $a$ is not
uniquely defined by the operator itself. Note that if $a$ belongs to
the space of all continuous functions $C=C(\sT)$, then the Hankel
operator $H(a)$ is compact on the space $H^p$. Moreover, Hankel
operators are never Fredholm, whereas if a Toeplitz operator
$T(a):H^p\to H^p$ is Fredholm, then $T(a)$ is one-sided invertible.
On the other hand, Fredholm Toeplitz operators with matrix-valued
generating functions are not necessarily one-sided invertible, and
this fact causes numerous difficulties in their studies.

A function $a\in L^\infty$ is called piecewise continuous if for
every $t\in \sT$ the one-sided limits $a(t+0)$ and $a(t-0)$ exist.
The set of all piecewise continuous functions is denoted by
$PC(\sT)$ or simply by $PC$. It is well-known that $PC$ is a closed
subalgebra of $L^\infty$. Any function from $PC$ has at most
countable set of jumps. Moreover, for each $\delta>0$ the set $S:=\{
t\in\sT:|a(t+0)-a(t-0)|>\delta \}$ is finite.

If $\cB$ is a unital subalgebra of $L^\infty$ containing the algebra
$C$, then we denote by $T(\cB)$  and $TH(\cB)$ the smallest closed
subalgebra of $\cL(H^p)$ containing all Toeplitz operators $T(a)$,
respectively, Toeplitz plus Hankel operators  $T(a)+H(b)$ with
generating functions from $\cB$. In the case where $\cB=C$ the
algebras $TH(C)$ and $T(C)$ coincide since the algebra $T(C)$
already contains all compact operators. Additionally, an operator
$T(a)+H(b)$ is Fredholm if and only if $a(t)\neq 0$ for all $t\in
\sT$ and
\begin{equation*}
\ind (T(a)+H(b))=\ind T(a)=-\wind a,
\end{equation*}
where $\wind a$ is the winding number of the contour $\Gamma:=\{
a(t):t\in \sT\}$ with respect to the origin. The algebras $T(PC)$
and $TH(PC)$ are much more complicated and will be discussed in
Section \ref{s5}.

The theory of Toeplitz plus Hankel operators is heavily based upon
the relations
 \begin{align}
   T(ab)=T(a)T(b)+H(a)H(\widetilde{b}), \label{eq1}\\
    H(ab)=T(a)H(b)+H(a)T(\widetilde{b}), \label{eq2}
\end{align}
which show that Toeplitz and Hankel operators are closely related
even for general $a,b\in L^\infty$. A consequence of \eqref{eq1} is
that for $a=c_- c c_+$ with $c_-\in \overline{H^\infty}, c\in
L^\infty, c_+\in H^\infty$, the representation $T(a)=T(c_-)T(c)
T(c_+)$ holds. This representation is often  used in the forthcoming
sections. More facts about Toeplitz and Hankel operators can be
found in \cite{BS, GK1992a, GK1992b}.

 \section{A method to study invertibility in $TH(L^\infty)$}

To study the invertibility of the operators $T(a)+H(b)$ acting on
the space $H^p$, $1<p <\infty$ we shall make use of a well-known
classical formula. Let $X$ and $Y$  be arbitrary operators acting on
the space $L^p$ and let $J$ be the above defined flip operator. Then
the relation
 \begin{equation}\label{eq3}
\left(%
\begin{array}{cc}
  X & Y \\
  JYJ & JXJ \\
\end{array}
 \right)%
=\frac{1}{2} \left(%
\begin{array}{cc}
  I & I \\
  J& -J \\
\end{array}%
\right)
 \left(%
\begin{array}{cc}
  X+YJ & 0 \\
 0 & X-YJ \\
\end{array}%
\right)
 \left(%
\begin{array}{cc}
  I & J \\
  I & -J \\
\end{array}
\right)
\end{equation}
holds, and the factors
 \begin{equation*}
\frac{1}{2}\left(%
\begin{array}{cc}
  I & I \\
  J& -J \\
\end{array}%
\right )  ,
 \quad
\left(%
\begin{array}{cc}
  I & J \\
  I & -J \\
\end{array}
\right).
\end{equation*}
are the inverses to each other. This identity is widely used to
study singular integral operators with conjugation and also
equations with  Carleman shifts preserving orientation.

Note that the operators $T(a)\pm H(b):H^p\mapsto H^p$ are invertible
or Fredholm if and only if so are the operators $T(a)\pm H(b)+Q=
(PaP+Q)\pm PbQ J:L^p\mapsto L^p$. Thus we can study the modified
operator $T(a)\pm H(b)+Q$ instead of the initial one, and if we set
$X=PaP+Q$, $Y=PbQ$, then relation \eqref{eq3} leads to the
representation
  \begin{align}
\cA:= & \left(%
\begin{array}{cc}
  PaP+Q & PbQ \\
 Q\widetilde{b}P & Q\widetilde{a}Q + P \\
\end{array}%
\right)=
\left(%
\begin{array}{cc}
  PaP+Q & PbQ \\
 JPbQJ & J(PaP+Q)J \\
\end{array}%
\right)  \label{eq4} \\[2ex]
 = &\frac{1}{2} \left(%
\begin{array}{cc}
  I & I \\
  J& -J \\
\end{array}%
\right)
\left(%
\begin{array}{c@{\hspace{-1mm}}c}
  PaP+Q+PbQJ & 0 \\
  0 & PaP+Q-PbQJ \\
\end{array}%
\right)
 \left(%
\begin{array}{cc}
  I & J \\
 I & -J \\
\end{array}
\right) . \nn
\end{align}
Thus the operator $\diag (T(a)+ H(b),T(a)- H(b))$ acting on the
space $H^p\times H^p$ has the same Fredholm properties as the
operator $\cA$ acting on $L^p\times L^p,\, 1<p<\infty$. In
particular, the kernels of these two operators have the same
dimension and so are the cokernels. Let us also recall that the
shift $J$ changes the orientation and this fact causes essential
complications, for example, it can happen that one of the operators
$T(a)\pm H(b)$ is Fredholm but the other one is not. Examples where
such situation occurs are presented in Section \ref{s5}, cf.
Examples \ref{ex1}--\ref{ex2}.

From now on we always assume that the generating function $a$ is
invertible. This does not lead to loss of generality since
Fredholmness of the operator $T(a)+H(b)$ implies the invertibility
of $a$ in $L^\infty$, \cite {BE2004,RS1987}. Now we can formulate
the first result.

 \begin{theorem}\label{t1}
The operators $\diag (T(a)+H(b),T(a)-H(b)): H^p\times H^p \mapsto
H^p\times H^p$ and $T(U(a,b)): H^p\times H^p \mapsto H^p\times H^p$,
 \begin{equation}\label{eq5}
 U(a,b)=\left(%
\begin{array}{cc}
  a-b \widetilde{b}\widetilde{a}^{-1} & - b\widetilde{a}^{-1}\\
   \widetilde{b} \widetilde{a}^{-1} &  \widetilde{a}^{-1} \\
   \end{array}%
\right)
\end{equation}
are simultaneously Fredholm or not. If they are Fredholm, then
\begin{equation}\label{eq6}
\ind\diag (T(a)+H(b),T(a)-H(b))=\ind T(U(a,b)).
\end{equation}
Moreover the kernel and cokernel dimensions of these operators
coincide.
 \end{theorem}

 \begin{proof}
Let us first represent the operator $\cA$ in a different form. To
this end we recall the following fact. If $p$ is an idempotent in a
unital Banach algebra $\cC$ with identity $e$, then the elements
$ap+(e-p)$ and $pap+(e-p)$ have identical invertibility properties.
Consequently, if $\cC=\cL(Z)$, where $Z$ is a Banach space, then
$ap+(e-p)$ and $pap+(e-p)$ have identical Fredholm properties. This
immediately follows from the relation
 \begin{equation}\label{eq7}
ap+(e-p)=(pap+(e-p))(e+(e-p)ap),
\end{equation}
and from the invertibility of the element $e+(e-p)ap$ with
$((e+(e-p)ap)^{-1}=e-(e-p)ap$. Consider now the idempotent $p\in
\cL(L^p\times L^p)$ defined by $p:=\diag (P,Q)$. Then the operator
$\cA$ of \eqref{eq4} can be represented in the form \cite{ES1998},
 \begin{equation*}
\cA=p\left(%
\begin{array}{cc}
  a & b \\
  \widetilde{b}  & \widetilde{a} \\
   \end{array}%
\right)
  p +(\diag(I,I)-p).
 \end{equation*}
The relation \eqref{eq7} implies that the operators $\cA$ and
$\widehat{\cA}$,
 \begin{equation*}
\widehat{\cA}=\left(%
\begin{array}{cc}
  a & b \\
  \widetilde{b}  & \widetilde{a} \\
   \end{array}%
\right ) \diag(P,Q)+\diag (Q,P)
 \end{equation*}
have identical Fredholm properties. On the other hand, the operator
$\widehat{\cA}$ can also be written as
\begin{equation*}
\widehat{\cA}=\left(%
\begin{array}{cc}
  a & 0 \\
  \widetilde{b}  & 1 \\
   \end{array}%
\right ) \diag(P,P)+\left(%
\begin{array}{cc}
  1 & b \\
  0  & \widetilde{a} \\
   \end{array}%
\right )\diag (Q,Q).
\end{equation*}
Now we note that the invertibility of the function $a$ implies the
invertibility of the matrix
 \begin{equation*}
R:=\left(%
\begin{array}{cc}
  1 & b \\
  0  & \widetilde{a} \\
   \end{array}%
\right ).
 \end{equation*}
Therefore, applying the relation \eqref{eq7} again but to the
operator $R^{-1}\widehat{\cA}$, we  obtain the claim.
 \end{proof}

  \begin{remark}\label{r1}
If the operator $T(U(a,b))$ is Fredholm, then its kernel and
cokernel dimensions can be expressed via the partial indices of the
matrix $U(a,b)$. However, the computation of the partial indices is
a very complicated task, and nowadays there are only a few classes
of matrices the partial indices of  which can be determined
\cite{LS1987}. In the present paper we derive some results on
one-sided invertibility of Toeplitz plus Hankel operators but our
method is not directly linked to the partial indices approach.
  \end{remark}

Let us make an important assumption, viz. let us assume that the
generating functions of the corresponding Toeplitz plus Hankel
operator $T(a)+H(b)$ satisfy the condition
\begin{equation}\label{eq8}
a (t) \widetilde{a}(t)=b (t)\widetilde{b}(t)=1, \quad t\in\sT.
\end{equation}
Equation \eqref{eq8} is called the matching condition, and the
corresponding pair of functions $(a,b)$ is called  matching pair.
Note that the set of all matching pairs $(a,b)\in L^\infty\times
L^\infty$ endowed with the operation $(a_1,b_1)(a_2,b_2):=(a_1
a_2,b_1 b_2)$ is a group. For any matching pair $(a,b)$, consider
the functions $c:=a b^{-1}$ and $d:=b\widetilde{a}^{-1}$ and call
them the matching functions for $(a,b)$ or, simply, the matching
functions. Moreover, the duo $(c,d)$ is referred to as  the
subordinated matching pair for the pair $(a,b)$.  Obviously, any
matching function possesses the property
\begin{equation}\label{eq9}
 c (t)\widetilde{c}(t)=d (t)\widetilde{d}(t)=1, \quad t\in \sT.
\end{equation}
On the other hand, if a function $c$ satisfies the equation
\eqref{eq9}, then it is the matching function for any pair $(ac,a)$
and for any pair $(a,a\widetilde{c})$, $a\in L^\infty$. Therefore,
in the following any function satisfying \eqref{eq9} is called
matching function. Moreover, if $c_1, c_2$ are matching functions,
then the product $c_1 c_2$ is the one as well.

An example of a matching function is the function $c=c(t)=t^n, t\in
\sT, n\in \sZ$. The set of matching functions  is quite large. For
instance, let an element $g_0\in L^\infty$ be continuous at the
points $t=\pm 1$, invertible in $L^\infty$ and such that
$g_0(\pm1)\in \{-1,1\}$. Let $\sT_+:=\{t\in \sT:\Im t\geq 0 \}$ be
the upper half-circle. Set
 \begin{equation*}
g(t)=\left \{
 \begin{array}{ll}
g_0(t) & \text{ if } t\in \sT_+ \\
g^{-1}_0(\overline{t})& \text{ if } t\in \sT\setminus \sT_+
 \end{array}
 \right. .
 \end{equation*}
Then $g\widetilde{g}=1$ for all $t\in \sT$, so $g$ is a matching
function.

In passing note that if $(a,b)$ is a matching pair, then the matrix
$U(a,b)$ takes the form
\begin{equation*}
U(a,b)=
\left(%
\begin{array}{cc}
 0  & -b \widetilde{a}^{-1} \\
 \widetilde{b}^{-1} \widetilde{a}^{-1}   & \widetilde{a}^{-1} \\
   \end{array}%
\right) =\left(%
\begin{array}{cc}
0   & - d  \\
 c   & \widetilde{a}^{-1} \\
   \end{array}%
\right).
\end{equation*}
Thus, in this case $U(a,b)$ is a triangular matrix and the
corresponding Toeplitz plus Hankel operator can be studied in more
detail.


Now we need a few more facts concerning the connections between the
Fredholmness of certain matrix operators and their entries.  Let
$\cR$ be an associative ring with the unit $e$. Denote by
$\cR^{2\times2}$ the set of all $2\times 2$ matrices with entries
from $\cR$. This set, equipped with usual matrix operations, forms a
non-commutative ring with the unit $E=\diag (e,e)$.

 \begin{proposition}\label{p3}
Let $\alpha,\beta,\gamma \in \cR$. Then
 \begin{enumerate}
\item If the matrix
 \begin{equation*}
\Upsilon :=
\left(%
\begin{array}{cc}
  0 & -\beta \\
   \alpha & \gamma  \\
   \end{array}%
\right)
 \end{equation*}
is invertible, then the elements $\alpha$ and $\beta$ are,
respectively, invertible from the left and from the right. If, in
addition, one of these two elements is invertible, then so is the
other.

 \item If both elements $\alpha$ and $\beta$ are invertible, then
 the matrix $\Upsilon$ is invertible too.
 \end{enumerate}
 \end{proposition}

  \begin{proof}
The matrix $\Upsilon$ can be represented in the form
\begin{equation*}
\left(%
\begin{array}{cc}
  0 & -\beta \\
   \alpha & \gamma  \\
   \end{array}%
\right)=
\left(%
\begin{array}{cc}
  \beta & 0 \\
   0 & e  \\
   \end{array}%
\right)
\left(%
\begin{array}{cc}
  0 & -e \\
   e & \gamma  \\
   \end{array}%
\right)
\left(%
\begin{array}{cc}
 \alpha  & 0 \\
   0 & e  \\
   \end{array}%
\right),
\end{equation*}
which immediately leads to the result desired since the middle
factor in the left-hand side is invertible.
  \end{proof}

To prepare the next theorem we need few more notation.  If $u,v\in
L^\infty$ are functions such that the operators $T(u)$ and $T(v)$
are Fredholm, and hence one-sided invertible, we denote by
$(\kappa_1,\kappa_2)$ the pair $(\ind T(v), \ind T(u))$. Let $\sZ_+$
and $\sZ_-$ be, respectively, the sets of all non-negative and
negative integers. Finally, if $T(\psi)$, $\psi\in L^\infty$ is
Fredholm, then $\psi$ can be represented in the form $\psi=\psi_0
t^n$, where the operator $T(\psi_0)$ is invertible and $n=-\ind
T(\psi)$. Thus for $n\leq 0$, a right inverse for the operator
$T(\psi)=T(t^n)T(\psi_0)$ is
$T_r^{-1}(\psi)=T^{-1}(\psi_0)T(t^{-n})$. If $n\geq 0$, a left
inverse for $T(\psi)=T(\psi_0)T(t^n)$ is
$T_l^{-1}(\psi)=T(t^{-n})T^{-1}(\psi_0)$. Given functions $u,v,w\in
L^\infty$, let us consider the matrix function $G$ defined by
 $$
G:=\left(%
\begin{array}{cc}
  0 & -v  \\
  u  & w  \\
   \end{array}%
\right).
 $$

 \begin{theorem}\label{t2}
Suppose that one of the operators $T(u)$ or $T(v)$ is Fredholm. Then
the Toeplitz operator $T(G)$ is Fredholm if and only if both
operators $T(u)$ and $T(v)$ are Fredholm. If they are Fredholm, then
 \begin{equation*}
\ind T(G)=\ind\diag(T(v),T(u)) = \ind T(v)+\ind T(u).
\end{equation*}
Moreover,
 \begin{enumerate}
\item If $(\kappa_1,\kappa_2)$ belongs to one of the sets
$\sZ_+\times \sZ_+$, $\sZ_-\times \sZ_+$, $\sZ_-\times \sZ_-$, then
 \begin{equation}\label{ind1}
 \dim\ker T(G)= \dim\ker \diag(T(v),T(u)).
\end{equation}
 \item If $(\kappa_1,\kappa_2)\in \sZ_+\times \sZ_-$, then
\begin{align}\label{ind2}
 \dim\ker T(G) &= \dim\ker \left (P_{-\kappa_2-1} T^{-1}(u_0) T(w)\left|_{\ker T(v)}\right
 . \right )\\
   & = \dim\ker \left (P_{-\kappa_2-1} T^{-1}(u_0) T(w)T^{-1}(v_0)\left|_{\im
   P_{{\kappa_1-1}}}\right . \right )
   \nn
   \end{align}
with $P_{m-1}:=I-T(t^m)T(t^{-m}),\, m\geq 1$.
 \end{enumerate}
 \end{theorem}

\begin{proof}
It is clear that if $T(G)$ and one of the operators $T(u)$, $T(v)$
is Fredholm, then so is $\diag(T(v),T(u))$. Really, since the
operator $T(G)$ can be represented in the form
  \begin{equation}\label{eqnG}
T(G)=
 \left(%
\begin{array}{cc}
  T(v) &  0\\
 0   &  I \\
   \end{array}%
\right)
\left(%
\begin{array}{cc}
 0 &-I  \\
  I   &  T(w)\\
   \end{array}%
\right)
\left(%
\begin{array}{cc}
 T(u)  &  0\\
 0   &  I\\
   \end{array}%
\right),
\end{equation}
and the middle factor in the right-hand side of \eqref{eqnG} is
invertible, the claim follows. On the other hand, if $T(G)$ is not
Fredholm then, according to \eqref{eqnG}, one of the operators
$T(u)$, $T(v)$ is also not Fredholm. In order to establish identity
\eqref{ind1}, one has to employ the representation \eqref{eqnG} once
more.

 The proof of the relation \eqref{ind2} is more complicated. To this end one has
 to study the set $S_1:=\ker (\diag (T(v),I))\cap \im A_1$, where
  $$
A_1 =
\left(%
\begin{array}{cc}
 0 & -I \\
 I   &  T(w)\\
   \end{array}%
\right)
\left(%
\begin{array}{cc}
  T(u) & 0 \\
  0  & I \\
   \end{array}%
\right)=
 \left(%
\begin{array}{cc}
 0  &  -I\\
  T(u)  & T(w) \\
   \end{array}%
\right).
   $$
Let $T_l^{-1}(u)$ denote one of the left inverses of the operator
$T(u)$, and let $A_2$ be the matrix operator defined by
 $$
A_2:=\left(%
\begin{array}{cc}
  T_l^{-1}(u) T(w) & I \\
  -I  &  0\\
   \end{array}%
\right).
 $$
Then
 $$
A_1 A_2= \left(%
\begin{array}{cc}
  I &  0\\
  (T(u) T_l^{-1}(u)-I)T(w)  & T(u) \\
   \end{array}%
\right),
 $$
 and
  $$
\left(%
\begin{array}{cc}
  I & 0 \\
  0& T_l^{-1}(u)   \\
   \end{array}%
\right)
 A_1 A_2= \left(%
\begin{array}{cc}
 I  & 0 \\
  0  & I \\
   \end{array}%
\right),
  $$
so the operator $A_1 A_2$ is left-invertible. Since the operator
$A_2$ is invertible, one obtains that  $\im A_1=\im(A_1A_2)$. It
follows from the theory of one-sided invertible operators
\cite{GK1992a} that the operator $R$,
 $$
R:= A_1 A_2 \diag (I, T_l^{-1}(u))  =
 \left(%
\begin{array}{cc}
  I &  0\\
  (T(u)T_l^{-1}(u)-I)T(w)  & T(u) T_l^{-1}(u) \\
   \end{array}%
\right)
 $$
is a projection onto $\im A_1 A_2 =\im A_1$. Moreover, the
projection $R$ has the same kernel as the operator $\diag (I,
T_l^{-1}(u))$. Note that $\ker(\diag(T(v),I))=(\ker T(v),0)^t$,
where $(\alpha, \beta)^t$ denotes the vector-column with the entries
$\alpha$ and $\beta$. For $x\in \ker T(v)$ the element $(x,0)^t$
belongs to the set $\im A_1$ if and only if $R((x,0)^t)=(x,0)^t$,
i.e. if
 $$
  x\in \ker \left ( (T(u)T_l^{-1}(u)-I)T(w)\left |_{\ker T(v)}\right . \right ).
 $$
Using now the above mentioned  representation
$T_l^{-1}(u)=T(t^{\kappa_2})T^{-1}(u_0)$, one obtains
 \begin{multline*}
( (T(u)T_l^{-1}(u)-I)T(w)\left |_{\ker T(v)}\right . \\=
    T(u_0)((T(t^{-\kappa_2})T(t^{\kappa_2})-I)T^{-1}(u_0)T(w))\left |_{\ker T(v)}\right
   . .
 \end{multline*}
 Because $T(u_0)$ is invertible, we get that $(x,0)^t$, $x\in \ker
 T( v)$ belongs to $\im A_1$ if and only if
    \begin{align*}
 &x\in \ker \left ( P_{-\kappa_2-1} T^{-1}(u_0) T(w)\left|_{\ker T(v)}
\right .  \right ) \\
 &\quad =     \ker \left ( P_{-\kappa_2-1} T^{-1}(u_0)
T(w)T^{-1}(v_0)\left|_{\im P_{\kappa_1-1})} \right .  \right ).
\end{align*}
 The theorem is proved.
 \end{proof}

 Set now $u=c$, $v=d$ and
 $w=\widetilde{a}^{-1}$.

  \begin{corollary}\label{c1}
Let $(a,b)\in L^\infty \times L^\infty$ be a matching pair with
matching functions $c$ and $d$. Then
 \begin{enumerate}
 \item If \/  $\ind T(c)\geq0$ and $\ind T(d)\geq 0$, then
 both  operators $T(a)+H(b)$ and $T(a)-H(b)$ are
 right-invertible.

 \item If \/  $\ind T(c)\leq0$ and $\ind T(d)\leq 0$, then
 both  operators $T(a)+H(b)$ and $T(a)-H(b)$ are
 left-invertible.
\end{enumerate}
 \end{corollary}

Moreover, employing Theorems \ref{t1} and \ref{t2}, one can also
describe the kernel dimensions of Toeplitz plus Hankel operators
$T(a)\pm H(b)$. Thus for $PC$-generating functions $a$ and $b$ the
following theorem is true.
 \begin{theorem}\label{tt2}
If $a,b\in PC$ form a matching pair with matching functions $c$ and
$d$, then the operators $\diag(T(a)+H(b),T(a)-H(b))$  and $\diag
(T(d), T(c))$ are simultaneously Fredholm or not. If they are
Fredholm, then
 \begin{align}\label{ind3}
 \ind \diag(T(a)+H(b),T(a)-H(b))& \\
   \quad =\ind \diag (T(d ), T(c))
 &=\ind T(d) + \ind
T(c).\nn
 \end{align}
Moreover, if $\kappa_1:=\ind T(d)$ and $\kappa_2:=\ind T(c)$, then
the kernels of the above mentioned operators are connected in the
following way.
 \begin{enumerate}
    \item If $(\kappa_1,\kappa_2)$ belongs to one of the sets
$\sZ_+\times \sZ_+$, $\sZ_-\times \sZ_+$, $\sZ_-\times \sZ_-$, then
 \begin{equation}\label{ind4}
 \dim\ker(\diag(T(a)+H(b),T(a)-H(b)))   = \dim \ker (\diag (T(d), T(c))).
\end{equation}
 \item If $(\kappa_1,\kappa_2)\in \sZ_+\times \sZ_-$, then
\begin{align}\label{ind5}
&\dim\ker\diag(T(a)+H(b),T(a)-H(b))=\\
   & = \dim\ker \left (P_{-\kappa_2+1} T^{-1}(c t^{-\kappa_2})
    T(\widetilde{b}^{-1}c)T^{-1}(dt^{\kappa_1})\left|_{\im
   P_{{\kappa_1-1}}}\right . \right ) \nn
     \end{align}
 \end{enumerate}
  \begin{proof}
Let us first point out that semi-Fredholm Toeplitz operators with
piecewise continuous generating functions are indeed Fredholm
operators. The latter statement can be easily derived from
Proposition 3.1 in Section 9.3 of \cite{GK1992b} in conjunction with
Lemma 1 in Section 18 of \cite{Mu2003}. The remaining part of the
proof immediately follows from Theorems \ref{t1} and \ref{t2}.

The last theorem will be used later in the study of invertibility of
Toeplitz plus Hankel operators with piecewise continuous generating
functions.  Note that necessary information on properties of such
operators, including an index formula, will be provided in Section
\ref{s5}.
  \end{proof}

 \end{theorem}

\section{$PC$-generating functions \label{s5}}

For the following, we need additional results concerning Toeplitz
and Toeplitz plus Hankel operators with piecewise continuous
generating functions acting on the space $H^p$, $1<p<\infty$. For,
let us introduce the functions
 \begin{align*}
 \nu_p(y)&:=\frac{1}{2}\left ( 1+\coth\left (\pi
\left(y+\frac{i}{p}\right ) \right )\right
    ), \quad 
 h_p(y):= \sinh^{-1}\left (\pi \left(y+\frac{i}{p}\right )
\right ),
\end{align*}
where $y\in \overline{\sR}$, and $\overline{\sR}$ refers to the
two-point compactification of $\sR$. Note that for  given points
$u,w\in \sC$, $u\neq w$ the set $\cA_p(u,w):=\{z\in\sC:
z=u(1-\nu_p(y))+w \nu_p(y), y\in \overline{\sR}\}$ forms a circular
arc which starts at $u$ and ends at $w$ as $y$ runs through
$\overline{\sR}$. This arc $\cA_p(u,w)$ has the property that from
any point of the arc, the line segment $[u,w]$ is seen at the angle
$2\pi/(\max\{p,q\})$, $1/p+1/q=1$. Moreover, if $2<p<\infty$
($1<p<2$) the arc $\cA_p(u,w)$ is located on the right-hand side
(left-hand side) of the straight line passing through the points $u$
and $w$ and directed from $u$ to $w$. If $p=2$, the set $\cA_p(u,w)$
coincides with the line segment $[u,w]$.

Let us list some properties of the function $h_p$. From the relation
 \begin{equation*}
\sinh \left ( \pi \left ( y+\frac{i}{p}\right )\right ) =\cos\left (
\frac{\pi}{p} \right ) \sinh (\pi y) +i\sin\left ( \frac{\pi}{p}
\right ) \cosh (\pi y),
 \end{equation*}
one easily obtains the following result.

  \begin{enumerate}
    \item The real (imaginary) part of $h_p$ is an odd (even)
function.
    \item The function $h_p: \sR \to \sC$ takes values in
the lower half-plane $\Pi^-:=\{z\in\sC:\Im z<0\}$ and
$h_p(\pm\infty)=0$.
    \item If $y$ runs from $-\infty$ to $+\infty$, the values of
$h_p$ trace out an oriented closed continuous curve. For $1<p\leq2$,
 \begin{equation*}
h_p(y)\in \left \{%
 \begin{array}{ll}
\Omega_1 & \text{ if } y\in[-\infty,0]\\
\Omega_2 & \text{ if } y\in[0,\infty]
 \end{array}
 \right .,
 \end{equation*}
 where $\Omega_1 :=\{z\in\sC: \Im z\leq 0 \text{ and } \RE z \geq 0
 \}$, $\Omega_2 :=\{z\in\sC: \Im z\leq 0 \text{ and } \RE z \leq 0
 \}$. On the other hand, for $2\leq p<\infty$,
 \begin{equation*}
h_p(y)\in \left \{%
 \begin{array}{ll}
\Omega_2 & \text{ if } y\in[-\infty,0]\\
\Omega_1 & \text{ if } y\in[0,\infty]
 \end{array}
 \right ..
 \end{equation*}
Thus for $p=2$ the point $h_p(y), y\in \sR$ moves along the
imaginary axis from the origin to the point $z_0=-i$ and then goes
back to the origin.
    \item For all $y\in\overline{\sR}$ the function $h_p$
    satisfies the inequality
     \begin{equation}\label{ineq}
\left | \Im h_p(y)\right |\leq \frac{1}{\sin\left ( \pi/p\right )},
     \end{equation}
and $y=0$ is the only point where equality holds.
  \end{enumerate}

As was already mentioned, $T(PC)$ and $TH(PC)$ are the smallest
closed subalgebras of $L(H^p)$ containing all Toeplitz operators
$T(a)$ and Toeplitz plus Hankel operators $T(a)+H(b)$  with $a,b\in
PC$, respectively. An index formula for Fredholm operators from
$T(PC)\subset \cL(H^p)$ is well-known and goes back to I.~Gohberg
and N.~Krupnik \cite{BS,GK1992a}. Let us now recall some results
from \cite{BS} and \cite{RSS2011}. First of all, we note that the
algebra $T(PC)$ contains the set $\cK(H^p)$ of all compact operators
from $\cL(H^p)$, and the Banach algebra $T^\pi(PC):=T(PC)/\cK(H^p)$
is a commutative subalgebra of the Calkin algebra the maximal ideal
space of which $\cM(T^\pi(PC))$ can be identified with the cylinder
$\sT\times [0,1]$ equipped with an exotic topology. However, in this
paper we prefer to identify the space of maximal ideal of the above
algebra with $\sT\times \overline{\sR}$, and by $\smb T^\pi(PC)$ we
denote the Gelfand transform of the algebra $T^\pi(PC)$.

 \begin{theorem}\label{t3}
 Let $T^\pi(PC)$ be the above defined Banach algebra. Then
  \begin{enumerate}
\item On the generators $T^\pi(a):=T(a)+\cK(H^p)$, $a\in PC$ the Gelfand transform
$\smb\!\!: T^\pi(PC)\mapsto C(\cM(T^\pi(PC)))$ is given  by
 \begin{equation*}
\smb T^\pi(a)(t,y)= a(t+0)\nu_p(y)+a(t-0)(1-\nu_p(y)), \quad
(t,y)\in \sT\times \overline{\sR}.
 \end{equation*}
\item The operator $A\in T(PC)$ is Fredholm if and only if
 \begin{equation*}
(\smb A^\pi)(t,y)\neq 0 \quad \text{for all} \quad (t,y)\in
\sT\times \overline{\sR}.
 \end{equation*}

 \item If $A\in T(PC)$ is Fredholm, then
  \begin{equation*}
\ind A=-\wind \smb A^\pi.
  \end{equation*}
  \end{enumerate}
 \end{theorem}

 \begin{remark}\label{r3}
The function $\smb A$ traces out an oriented curve, the orientation
of which is inherited from those of \/ $\overline{\sR}$ and $\sT$.
Thus if $\smb A$ does not cross the origin, the winding number
$\wind \smb A$ is well-defined.  We shall also agree on writing
$\smb A$ for $\smb A^\pi$, so we can think about the homomorphism
$\smb$ as  a function acting on the whole algebra $T(PC)$.  In this
case, $\smb A=0$ for any compact operator $A$.
\end{remark}

Later on we will see that for generating functions $b \in PC$ having
discontinuity points on $\sT\setminus \{-1,1\}$, the Hankel operator
$H(b)$ does not belong to the algebra $T(PC)$. On the other hand, we
need a non-trivial result that for $b\in PC\cap C(\sT\setminus
\{-1,1\})$ the operator $H(b)\in T(PC)$. Let us give an idea of the
proof for this fact. First of all, one can observe that it suffices
to show that for the characteristic function  $\chi_+$ of the upper
half-circle $\sT_+$, the corresponding Hankel operator $H(\chi_+)$
belongs to the Toeplitz algebra $T(PC)$. However, the last inclusion
could be verified following localization arguments from \cite[pp.
245--247]{RSS2011} and also using Theorem 4.4.5(iii) from there.
Moreover, it is established in \cite{RSS2011} that the Gelfand
transform of the coset $H^\pi(b)$, $b\in PC\cap C(\sT\setminus
\{-1,1\})$ is
\begin{equation*}
\smb H^\pi(b)(t,y)= \left\{%
\begin{array}{ll}
  \D \frac{b(1+0)-b(1-0)}{2} \,h_p(y) &  \text{ if }\,\, t=1\\[1.5ex]
   \D -\frac{b(-1+0)-b(-1-0)}{2}\, h_p(y) &  \text{ if }\,\, t=-1
   \\[1.5ex]
 0 & \text{ if }\,\, t\in \sT\setminus \{-1,1\}%
   \end{array}%
\right. ,
\end{equation*}
and
  \begin{equation*}
\ind(T(a)+H(b))=-\wind \smb (T(a)+H(b)).
\end{equation*}

The Fredholm theory for operators belonging to $TH(PC)$ is
considerably more complicated than that for $T(PC)$. This is due to
the fact that the Calkin image $TH^\pi(PC):=TH(PC)/\cK(H^p)$ of the
algebra $TH(PC)$ is not commutative. Let $\tp:=\{t\in \sT_+:\Im
t>0\}$.

 \begin{theorem}\label{t4}
 If $a,b\in PC$, then
 \begin{enumerate}
\item The operator $T(a)+H(b)$ is Fredholm if and only if the matrix
  \begin{align*}
    &\smb(T(a)+H(b))(t,y):= \nn\\[1ex]
  &  \left(%
\begin{array}{c@{\hspace{-1mm}}c}
a(t+0)\nu_p(y)+a(t-0)(1-\nu_p(y))   & \D \frac{b(t+0)-b(t-0)}{2i} \,h_p(y)\\
\D \frac{b(\overline{t}-0)-b(\overline{t}+0)}{2i} \,h_p(y)   &
 a(\overline{t}+0)\nu_p(y)+a(\overline{t}-0)(1-\nu_p(y))\\
   \end{array}%
\right) 
\end{align*}
is invertible for every $(t,y)\in \tp\times \overline{\sR}$ and the
function
 \begin{align*}
\smb (T(a)+H(b))(t,y):=& \, a(t+0)\nu_p(y)+a(t-0)(t-\nu_p(y))\\
 & +
\frac{b(t+0)-b(t-0)}{2} \,h_p(y)
\end{align*}
does not vanish on $\{-1,1\}\times \overline{\sR}$.

 \item The mapping $\smb$ defined in assertion (i) extends  for
each $(t,y)\in \sT_+\times \overline{\sR}$ to an algebra
homomorphism on the whole $TH(PC)$.

 \item An operator $A\in TH(PC)$ is Fredholm if and only if the
 function
  \begin{equation*}
  \smb A: \sT_+\times \overline{\sR}\mapsto \sC
  \end{equation*}
  is invertible.
 \end{enumerate}

 \end{theorem}

 \begin{remark}\label{r4}
\begin{enumerate}
\item Theorem \ref{t4} indicates that the algebra $TH^\pi(PC)$ is
not commutative. In particular, this means that Hankel operators
with generating functions from $PC(\sT)$ with discontinuity points
on $\sT\setminus \{-1,1\}$ do not belong to the algebra $T(PC)$.

\item This theorem also shows that the roles of the points $t=\pm1$
and $t\in\tp$ are very different. In connection with this, let us
remaind that $t=\pm1$ are the only fixed points of the operator $J$.

\item There is an index formula for Fredholm operators
$T(a)+H(b)$ considered on the space $l^p(\sZ_+)$ \cite{RS1}. The
approach of \cite{RS1} can be modified in order to get a similar
result for Toeplitz plus Hankel operators acting on $H^p$-spaces. On
the other hand, the index formula presented in Proposition \ref{p2}
below, is relatively simple and sufficiently effective to handle all
problems considered in this work.

 \item If $a\in PC$, then the symbols of the operator $T(a)$ defined
 by Theorems \ref{t3} and \ref{t4} coincide.
\end{enumerate}
 \end{remark}
Let us note that at the first time  Theorem \ref{t4} appeared in
\cite{RS1990}, see also \cite{RSS2011}.

We proceed with  additional preparatory material.

 \begin{theorem}\label{t5}
If $a,b\in PC$ are continuous at the points $t=-1$ and $t=1$, then
the operators $T(a)+H(b)$ and $T(a)-H(b)$ are simultaneously
Fredholm or not. If one of this operators is Fredholm, then
 \begin{equation*}
\ind(T(a)+H(b))=\ind(T(a)-H(b)).
\end{equation*}
Moreover,
 \begin{equation*}
\ind(T(a)+H(b))=\frac{1}{2}\ind T(U),
 \end{equation*}
 where $T(U)$ is the block Toeplitz operator on $H^p\times H^p$ with
 generating matrix-function $U=U(a,b)$ defined by \eqref{eq5}.
 \end{theorem}
The first assertion can be proved analogously to \cite{RS1}, see
also \cite{DS2012a}, and the second one directly follows from the
relation \eqref{eq6}.

\begin{remark}\label{r5}
In the case at hand, the index computation is reduced to the similar
problem for a block Toeplitz operator the generating matrix of which
has piecewise continuous entries. The index formula for such
operators is known \cite[Chapter 5]{BS}.
\end{remark}

As the next step we prove the following proposition.

\begin{proposition}\label{p1}
 If $a,b\in PC$ and the operator $T(a)+H(b)$ is Fredholm, then
 \begin{equation}\label{eq13}
 (T(a)+H(b))-(T(g)+H(b_0))(T(ag^{-1})+H(b_1g^{-1})) \in \cK(H^p),
\end{equation}
where $b_0$ is a function continuous on $\sT\setminus \{-1,1\}$ such
that $b_0(1\pm0)=b(1\pm0)$, $b_0(-1\pm0)=b(-1\pm0)$, $b_1:=b-b_0$,
and $g$ is a function on $\sT$, which is invertible, continuous on
$\sT\setminus \{-1,1\}$ and $a(1\pm0)=g(1\pm0)$,
$a(-1\pm0)=g(-1\pm0)$. Moreover, the operators $T(g)+H(b_0)$ and
$T(ag^{-1})+H(b_1g^{-1})$ are also Fredholm and $T(g)+H(b_0)$
belongs to the algebra $T(PC)$.
 \end{proposition}

\begin{proof} By Theorem \ref{t4}, the function $a$ is invertible in
$PC$. The existence of functions $b_0$ and $g$ with the properties
prescribed is obvious. Notice that $ag^{-1}$ is continuous at the
points $t=\pm1$ and $ag^{-1}(\pm1)=1$. Using the relations
\eqref{eq1}-\eqref{eq2} and Corollary 5.33 from \cite{BS}, one
arrives at  the representation \eqref{eq13}. The Fredholmness of all
the Toeplitz plus Hankel operators appeared is clear, and the
inclusion $T(g)+H(b_0)\in T(PC)$ for such a kind operators has been
mentioned before.
 \end{proof}

Now we are able to establish a transparent index formula for the
operators under consideration.

 \begin{proposition}\label{p2}
If \/ $T(a)+H(b)$ is Fredholm, then
\begin{equation*}
\ind(T(a)+H(b))=-\wind \smb(T(g)+H(b_0))+ \frac{1}{2}\ind T(U_1),
\end{equation*}
where
\begin{equation*}
U_1 =\left(%
\begin{array}{cc}
  a_{2}-b_{2} \widetilde{b}_{2}\widetilde{a}^{-1}_{2} & - b_{2}\widetilde{a}^{-1}_{2}\\
   \widetilde{b}_{2} \widetilde{a}^{-1}_{2} &  \widetilde{a}^{-1}_{2} \\
   \end{array}%
\right)
\end{equation*}
and $a_2=ag^{-1}$, $b_2=b_1 g^{-1}$.
 \end{proposition}

 \begin{proof} Remark that the function $b_2$ vanishes at the
 points $t=-1$ and $t=1$. The result now follows from the
 representation \eqref{eq13} and Theorem \ref{t5}.
  \end{proof}

 \begin{remark}\label{r6}
Proposition \ref{p2} indicates that if both operators $T(a)\pm H(b)$
are Fredholm, the difference between their indices  depends only on
the index difference for the operators $T(g)+H(b_0)$ and
$T(g)-H(b_0)$. However, one can observe that
 \begin{equation*}
\ind(T(g)+H(b_0)) - \ind(T(g)-H(b_0)) \in \{-2,-1,0,1,2\}.
 \end{equation*}
 \end{remark}

\begin{theorem}\label{t6}
Let $(a,b)\in PC(\sT)\times PC(\sT)$  be a matching pair with
matching functions $c$ and $d$. Then the following assertions hold.
 \begin{enumerate}
    \item The operators $T(a)\pm H(b)$ are Fredholm if and only if so are the
    Toeplitz operators $T(d)$ and $T(c)$.

    \item  The operators $T(a)\pm H(b)$ are invertible from the right if
     \begin{equation*}
\ind T(c)\geq 0 \quad \text{and} \quad \ind T(d)\geq 0.
     \end{equation*}

    \item  The operators $T(a)\pm H(b)$ are invertible from the left if
    \begin{equation*}
\ind T(c)\leq 0 \quad \text{and} \quad \ind T(d)\leq
 0.
    \end{equation*}

    \item If both numbers $m= \dim \ker T(U(a,b))$ and $n\!=\!\dim
\coker T(U(a,b))$ are not equal to zero but $\ind (T(a)+ H(b))=\ind
(T(a)- H(b))$, then at least one of the operators $T(a)+H(b)$ or
$T(a)-H(b)$ is not one-sided invertible.
 \end{enumerate}
\end{theorem}

 \begin{remark}\label{r7}
The index computation is reduced to the verification of the
condition $\ind (T(d)+ H(b_0))=\ind (T(d)- H(b_0))$, what can be
done effectively.
 \end{remark}

Notice than certain deficiency of Theorem \ref{t6} consists in its
inability to handle one-sided invertibility in the case where one of
the operators, say $T(a)+H(b)$, is Fredholm but $T(a)-H(b)$ is not.
This situation can still be studied but to do so we need a result of
I.~Shneiberg \cite{Sh1974}.

Let $A$ be an operator defined on each space $L_p$, $1<p<\infty$.
Then the set $A_{F}:=\{p\in(1,\infty):A \quad \text{is Fredholm}\}$
is open. Moreover, for each connected component of the set $A_{F}$,
the index of $A$ is constant.

 Let us now assume that on a space $H^p$ the operator $T(a)+H(b)$
is Fredholm but $T(a)-H(b)$ is not. For the operator $T(a)+H(b)$,
consider  that  connected component of the set $(T(a)+H(b))_{F}$
which contains $p$. Consider also the operators $T(c)$ and $T(d)$.
If $T(a)-H(b)$ is not Fredholm, then at least one of the operators
$T(c)$ and $T(d)$ is not Fredholm. Replacing $p$ by $s$, where $s>p$
and is close enough to $p$, one obtains that $T(c)$ and $T(d)$ are
already Fredholm on $H^{s}$. Consequently, there is an interval
$(p,p')$, $p<p'$ such that $T(c)$ and $T(d)$ are Fredholm operators
for all $s\in (p,p')$. A result of I.~Shneiberg ensures us that
$\ind T(c)\left |_{H^s \mapsto H^s}\right .$ and $\ind T(d)\left
|_{H^s \mapsto H^s}\right .$ are constant functions in $s$. Applying
Theorem \ref{tt2}, one obtains that this is also  true for the
operators $T(a)+H(b)$ and $T(a)-H(b)$ acting on $H^s$, $s\in
(p,p')$. Moreover, the above argumentation shows that this statement
is true even for $T(a)+H(b)$ acting on $H^s$ for all $s\in(p'',p')$
where $p''<p$ and is close enough to $p$.

Let us now formulate the final result.

 \begin{theorem}\label{t7}
Let  $(a,b)\in PC(\sT)\times PC(\sT)$  be a matching pair with
subordinated pair $(c,d)$, and let the operator $T(a)+H(b)$ be
Fredholm on a space $H^p, 1<p<\infty$. Then
\begin{enumerate}
    \item The operator $T(a)+H(b)$ is invertible from the right if
     \begin{equation*}
 \lim_{s\to p+}\ind T(c) \geq 0 \quad \text{and} \quad \lim_{s\to p+} \ind T(d)\left |_{H^s \mapsto H^s}\geq 0 \right ..
     \end{equation*}
    \item The operator $T(a)+H(b)$ is invertible from the left if
     \begin{equation*}
 \lim_{s\to p+}\ind T(c) \leq 0 \quad \text{and} \quad \lim_{s\to p+} \ind T(d)\left |_{H^s \mapsto H^s}\leq 0\right ..
     \end{equation*}
\end{enumerate}
 \end{theorem}

 \begin{proof} If $s\in (p, p')$, one can apply Theorem \ref{t6}
to obtain that on $H^s$ the operator $T(a)+H(b)$ is invertible from
the left, respectively, from the right  if  the condition (i),
respectively, (ii) is satisfied. Because $\ind (T(a)+H(b))\left
|_{H^s \mapsto H^s}\right .$ is constant on the interval $[p,p')$,
one can employ the following result \cite[Chapter ]{GF1974}. Let
$X_1,X_2$ be Banach spaces such that $X_1$ is continuously and
densely embedded into $X_2$, and let $A$ be a linear continuous
Fredholm operator on both $X_1$ and $X_2$ which has the same index
on each space. Then
 \begin{equation*}
\ker A\left |_{X_1\mapsto X_1}\right .= \ker A\left |_{X_2\mapsto
X_2}\right .
 \end{equation*}
and
 \begin{equation*}
 \coker A\left |_{X_1\mapsto X_1}\right .= \coker A\left
|_{X_2\mapsto X_2}\right .
 \end{equation*}
Therefore, according to Theorem \ref{t6},  the operator $T(a)+H(b)$
is also one-sided invertible on $H^p$, which completes the proof.
 \end{proof}

 \begin{example}
Let $a\in PC$ and the operator $T(a)+H(a):H^p\mapsto H^p$ be
Fredholm. The matching condition is trivial, and the last theorem
shows that the operator $T(a)+H(a)$ is one-sided invertible. Hence,
we recovered  a result of E.~Basor and T.~Ehrhardt for $a\in PC$. In
the following the index of certain operators will be computed. These
computations use Theorem \ref{t3} and the results presented after
that theorem.
 \end{example}

\begin{example}\label{ex1}
Consider the function  $a=\exp(i\xi/4), \xi\in(0,2\pi)$. This
function is continuous on $\sT\setminus \{1\}$ and has a jump at the
point $t=1$, viz. $a(1+0)=1$ and $a(1-0)=i$. The operators
$T(a)+H(at)$ and $T(a)-H(at)$ are well-defined on every space $H^p,
1<p<\infty$. Since $b=at$ the matching condition \eqref{eq8} is
satisfied with the matching functions $c(t)= t^{-1}$ and
$d=a\widetilde{a}^{-1}t$. Note that  $a
\widetilde{a}^{-1}=-i\exp(i\xi/2)$. By Theorem \ref{t3} the operator
$T(a \widetilde{a}^{-1})$ is not Fredholm on $H^2$ because $a
\widetilde{a}^{-1}(1+0)=-i$, $a \widetilde{a}^{-1}(1-0)=i$, and the
interval connecting these two points includes the origin. In view of
Theorem \ref{tt2}, at least one of the operators $T(a)+H(at)$ or
$T(a)-H(at)$ is not Fredholm. Since $H(ta)\in T(CP)$, it follows
from Theorem \ref{t3} that $T(a)-H(at)$ is Fredholm  but
$T(a)+H(at)$ is not. Moreover, $\ind(T(a)-H(at))=0$, and easy
arguments show  that $T(a)-H(at)$ is Fredholm on each space $H^p$:
If $p\neq 2$, then $T(a\widetilde{a}^{-1})$ is Fredholm with index
$-1$ for $2<p<\infty$ and with index $0$ for $1<p<2$. Then,
according to Theorem \ref{tt2}, the operator $T(a)-H(at)$ is
Fredholm for each $p\in (1,\infty)$. The above mentioned theorem of
I.~Shneiberg then entails that $T(a)-H(at)$ has index zero for all
$p\in (1,\infty)$. Thus we have
 \begin{equation}\label{eqI}
 \ind T(d)
= \left \{
  \begin{array}{ll}
 -1 & \text{ if } 1<p<2\\
  -2& \text{ if } 2<p<\infty
   \end{array}
    \right . ,
 \end{equation}
 whereas
  \begin{equation*}
\ind T(c)=\ind T(t^{-1})=1 \text{ for all } p\in (1, \infty).
  \end{equation*}
From Theorem \ref{tt2} one concludes that
 \begin{equation*}
\dim\ker \diag (T(a)+H(at),T(a)-H(at)) =1 \text{ for all } p\in
(1,\infty)\setminus \{2\},
 \end{equation*}
and
\begin{equation*}
\dim\coker \diag (T(a)+H(at),T(a)-H(at)) =\left \{
  \begin{array}{ll}
 -1 & \text{ if } 1<p<2\\
  -2& \text{ if } 2<p<\infty
   \end{array}
    \right . .
 \end{equation*}
Moreover,  the constant function $x(t)=1, t\in \sT$ belongs to the
kernel of the operator $T(a)-H(at)$. Therefore,
 \begin{equation*}
\dim\coker (T(a)-H(at))=1,
 \end{equation*}
 so this operator is not one-sided invertible. This leads to the
conclusion that $T(a)+H(at)$ is invertible if  $1<p<2$, and it is
invertible from the left with $\dim\coker(T(a)+H(at))=1$ if
 $2<p<\infty$.

In conclusion of this example, we would like to mention a quite
remarkable fact. It turns out that the approach presented here
allows one to find the kernels of Toeplitz plus Hankel operators in
various situations. Thus, let us consider the operators
$T(a)+H(at^n)$ and $T(a)-H(at^n)$ where $a$ is the above defined
function and $n\in \sZ_+$. The reader can verify that if $n=2m+1,
m\in \sZ_+$, then
 \begin{align*}
\ker(T(a)+H(at^n)) & = \left \{ t^{m+k}-t^{m-k}, \, k=1,2, \ldots, m\right \}, \\
\ker(T(a)-H(at^n)) & = \left \{ t^{m},t^{m+k}+t^{m-k} , \, k=1,2,
\ldots, m\right \},
\end{align*}
 and if $n=2m, m\in \sZ_+$, then
 \begin{align*}
\ker(T(a)+H(at^n)) & = \left \{ t^{m-k-1}-t^{m+k}, \, k=0,1,2, \ldots, m-1\right \}, \\
\ker(T(a)-H(at^n)) & = \left \{ t^{m-k-1}+t^{m+k} , \, k=0,1,2,
\ldots, m-1\right \}.
\end{align*}
It is also worth noting that for arbitrary function $a\in PC$, the
above presented functions belong to the kernel of the corresponding
operators  $T(a)+H(at^n)$ and $T(a)-H(at^n)$ but they may not
exhaust it.
\end{example}

 \begin{example}\label{ex3}
 Consider the operator $T(a)+H(at^{-1})$ where $a$ is the function
 given in Example \ref{ex1}. It is clear that $(a, at^{-1})$ is a
 matching pair with the matching functions $c(t)=t$ and $d=a \widetilde{a}^{-1} t^{-1}$.
 Using \eqref{eqI}, one obtains
    \begin{equation*}
\ind T(d) = \left \{
  \begin{array}{ll}
 1 & \text{ if } 1<p<2\\
  0& \text{ if } 2<p<\infty
   \end{array}
    \right . .
 \end{equation*}
Besides for $p=2$ the operator $T(a)-H(at^{-1})$ is Fredholm but
$T(a)+H(at^{-1})$ is not. Further, since $\ind
(T(a)-H(at^{-1}))\left|_{H^2\mapsto H^2} \right .=0$, the operator
$T(a)-H(at^{-1})$ is Fredholm on all spaces $H^p$, $1<p<\infty$ with
the index $0$, cf. Example \ref{ex1}. Now Theorem \ref{tt2}(i)
entails that if $2<p<\infty$, then the operator
$T(a)-H(at^{-1}):H^p\mapsto H^p$ is invertible whereas
$T(a)+H(at^{-1})$ is left-invertible with codimension one.

 To study the properties of the operators $T(a)\pm H(at^{-1})$
in more detail for $1<p<2$, let us first note that we are in the
situation described in Theorem \ref{tt2}, case (ii). Further, on
each space $H^p$, $1<p<2$ the operators $T(\widetilde{a}^{-1})$ and
$T(a\widetilde{a}^{-1})$ are Fredholm with the index zero, so they
are invertible.

Note that $P_0 T(\widetilde{a}^{-1})T(a\widetilde{a}^{-1}) P_0:\im
P_0\to \im P_0$ is a non-zero operator. Really, let us show that
\begin{equation}\label{eqNZ}
   P_0 T(\widetilde{a}^{-1})T(a\widetilde{a}^{-1}) P_0:\im
P_0 (1)\neq 0.
\end{equation}
The proof of \eqref{eqNZ} can be given from the Wiener-Hopf
factorization of power functions. Definition and properties of such
a factorization can be found in \cite{BS, GK1992a,GK1992b}. Here we
will only sketch the proof. The details are left to the reader.

\sloppy

 Let us  introduce few more notation. From now on, the
argument $\arg z$ of a complex number $z\neq 0$ is always chosen so
that $\arg z\in [0,2\pi)$. If $\beta\in \sC$, the function
$\vp_\beta \in PC$ is defined by
 $$
\vp_\beta (\exp(i\zeta)):=\exp(\beta (\zeta-\pi)), \quad \zeta\in
[0,2\pi).
 $$
It is easily seen that the function $\vp_\beta$ has at most one
discontinuity - viz. a jump at the point $z_0=1$ so that
$\vp_\beta(1+0)=\exp(-\pi i \beta)$ and $\vp_\beta(1-0)=\exp(\pi i
\beta)$. Therefore, one has
 $$
\vp_\beta(\exp(i\zeta))=\exp(-i\beta \pi)\exp(i\beta \zeta),
 $$
and the representations
$a\widetilde{a}^{-1}(\exp(i\zeta))=-i\exp(i\zeta/2)$ and
$\widetilde{a}^{-1}(\exp(i\zeta))=-i\exp(i\zeta/4)$ imply that
\begin{equation}\label{eqN1}
a\widetilde{a}^{-1}=c_1 \vp_{1/2}, \quad \widetilde{a}^{-1}=c_2
\vp_{1/4},
\end{equation}
where $c_1,c_2\in \sC/\{0\}$.

Consider also the functions $\xi_\beta, \eta_\beta$ defined on
$\sT/\{0\}$ by
 \begin{align*}
\xi_\beta(t) &=\left ( 1-\frac 1t \right )^\beta :=\exp\left \{
\beta
  \log \left | 1-\frac1t \right | + i \beta \arg  \left ( 1-\frac1t \right )
  \right \}, \\
  \eta_\beta(t)&=(1-t)^\beta:=\exp\left \{ \beta
  \log  |1-t| + i \beta \arg ( 1-t )
  \right \}.
\end{align*}
Note that $\xi_\beta$ (resp., $\eta_\beta$) is the limit on the unit
circle $\sT$ of that branch of the function $(1-1/z)^\beta$ (resp.,
$(1-z)^\beta)$ which is analytic for $|z|>1$ (resp., $|z|<1$) and
takes the value $1$ at $z=\infty$ (resp., $z=0$). We also notice
that for all $t\in \sT/\{1\}$ the relations
\begin{equation*}
\xi_\alpha(t)\xi_\beta(t)=\xi_{\alpha+\beta}(t), \quad
\eta_\alpha(t)\eta_\beta(t)=\eta_{\alpha+\beta}(t),
 \end{equation*}
 and
 \begin{equation}\label{eqN2}
 \vp_\beta(t)=\xi_{-\beta}(t)\eta_\beta(t)
\end{equation}
hold. Moreover, taking into account the invertibility of the
operators $T(\vp_{1/4})$ and $T(\vp_{1/2})$ on the spaces $H^p$ for
$1<p<2$, one concludes that the identity \eqref{eqN2} represents a
Wiener-Hopf factorization of the function $\vp_\beta$ with respect
to $H^p$, $1<p<2$. Thus using \eqref{eqN2} and the relation
$P\xi_{1/2}P(1)=1$, one obtains
\begin{align*}
  P_0 T(\vp_{1/4}) T^{-1}(\vp_{1/2})(1)&= P_0 P \xi_{-1/4} P
  \eta_{1/4}P \eta_{-1/2}P\xi_{1/2} P(1)\\
   & = P_0 P \xi_{-1/4} \eta_{-1/4} (1).
\end{align*}
Now it is easily seen that
 $$
P_0 P \xi_{-1/4} \eta_{-1/4} (1)=c_0,
 $$
where $c_0$ is the Fourier coefficient at $t^0$ for the function
$\xi_{-1/4} \eta_{-1/4}\in L^1(\sT)$. In view of the expansions
 \begin{align*}
  \xi_{-1/4} &\!= \!1+\frac 14  z^{-1}+\frac{(1/4)(1+1/4)}{2!} z^{-2}+\frac{(1/4)(1+1/4)(2+1/4)}{3!}
  z^{-3}+ \cdots \, ,\\
 \eta_{-1/4} &= 1+\frac 14 z+\frac{(1/4)(1+1/4)}{2!} z^2+\frac{(1/4)(1+1/4)(2+1/4)}{3!}
  z^3+ \cdots \, ,
\end{align*}
we get
 $$
c_0=1^2+\sum_{k=1}^\infty \left (
\frac{(1/4)(1+1/4)\cdots(k-1+1/4)}{k!} \right )^2 \neq 0.
 $$
Theorem \ref{tt2}(ii) now leads to the conclusion that if $1<p<2$,
then
 $$
\dim\ker\diag (T(a)+H(t^{-1}a),T(a)-H(t^{-1} a))\left |_{H^p\to
H^p}\right .=0
 $$
Recalling that the index of this operator is zero, one immediately
obtains that on the space $H^p$, $1<p<2$, the operators
$T(a)+H(t^{-1}a)$ and $T(a)-H(t^{-1} a)$ are invertible. Moreover,
since for any $p$, $1<p<\infty$ the operator $T(a)-H(t^{-1}
a):H^p\to H^p$ is Fredholm, it is invertible on each space $H^p$,
$1<p<\infty$.
 \end{example}

 \begin{remark}
 The above approach can be used to obtain complete information
about the invertibility of more general operators $T(a)\pm
H(at^{-n})$, $n\in\sZ$. However, due to lack of space we do not
consider this problem here.
 \end{remark}

 \begin{example}\label{ex2}
Let $a$ be the following piecewise continuous function
 \begin{equation*}
a(t)=\left \{%
 \begin{array}{rl}
 1 & \text{ if } t\in\tp \\
-1 & \text{ if } t\in\sT\setminus \sT_+
\end{array}
\right ..
 \end{equation*}
On the space $H^p$, $1<p<\infty$, consider the operator $iI+H(a)$.
The identity $i\cdot i=a \widetilde{a}=-1$ shows that the matching
condition is satisfied with the matching function $c=ia$.  Theorem
\ref{tt2} entails that for every $p\in(1,\infty)$ the operators
$\diag(iI+H(a), iI-H(a))$ and $\diag(-iT(a), iT(a))$ have identical
Fredholm properties on the space $H^p\times H^p$. Note that on the
space $L^2(\sT)$ the operator $T(a)$ is not Fredholm, so one of the
operators $iI+H(a)$ or  $iI-H(a)$ is also not Fredholm on $H^2$, cf.
Theorem \ref{tt2}. However, for  $p=2$ the range of the symbol of
the operator $iI-H(a)$ is located in the upper half-plane $\Pi^+:=\{
z\in\sC: \Im z>0 \}$, and  Theorem \ref{tt2} shows that the operator
$iI-H(a)$ is Fredholm on each space $H^p$. Moreover,
$\ind(iI-H(a))=0$ for any $p\in(1,\infty)$ whereas
$\ind(iI+H(a))=-2$ for $1<p<2$ and $\ind(iI+H(a))=2$ for
$2<p<\infty$. Really, this is a consequence of the fact that $T(a)$
is Fredholm with $\ind T(a)=1$ if $1<p<2$ and $\ind T(a)=-1$ if
$2<p<\infty$. Consequently, in view of Theorem \ref{tt2}, we have
$iI-H(a)$ is invertible for all $p\in (1,\infty)\setminus \{2\}$ but
$iI+H(a)$ is right or left invertible for $1<p<2$ and $2<p<\infty$,
respectively. The invertibility of the operator $iI-H(a)$ in the
space $H^2$ can be easily shown, so the proof is left to the reader.
 \end{example}

 \begin{example}
Consider the function $a\in PC$,
 \begin{equation*}
a(t)=\left \{%
 \begin{array}{rl}
1 & \text{ if } \re t\geq 0\\
-1 & \text{ if } \re t<0
  \end{array}
   \right .,
\end{equation*}
and the operator $T(a)+H(at)$ acting on the space $H^p$,
$1<p<\infty$. It is clear that $a=\widetilde{a}$, so $(a,at)$ is the
matching pair with matching functions $c(t)=t^{-1}$ and $d=t$. By
Theorem \ref{tt2}, one has
 \begin{equation*}
 \ind\diag (T(a)+H(at),T(a)-H(at))=0.
\end{equation*}
Along with Theorem \ref{t5}, this leads to the equation
 \begin{equation*}
 \ind(T(a)+H(at)) =\ind(T(a)-H(at))=0
\end{equation*}
for all spaces $H^p$. Since $1\in \ker (T(a)-H(at))$, one can use
Theorem \ref{tt2} and obtain that the operator $T(a)+H(at)$ is
invertible on all spaces $H^p$, whereas $T(a)-H(at)$ is not
one-sided invertible.
 \end{example}

 \begin{remark}
The previous examples show that if  additional information is
available, the invertibility problem can be studied completely, even
in the case where Theorem \ref{t6} is not working.
 \end{remark}

 \begin{remark}
Analyzing Example \ref{ex1}, one can establish the fact that if
$a\in L^\infty$ and the operator $T(a)+H(ta)$ is Fredholm, then it
is one-sided invertible.
 \end{remark}

 \begin{remark}
After this work was completed, the authors discovered an interesting
paper \cite{BE:2013}, where Toeplitz plus Hankel operators are
studied under the same assumption \eqref{eq8}. Nevertheless, our
approach is remarkably different from that of \cite{BE:2013}, and
the intersection of the theoretical results obtained is minimal.
 \end{remark}

\def\cprime{$'$} \def\cprime{$'$}



\end{document}